\documentclass[a4paper,11pt]{article}
\usepackage{hyperref}
\usepackage{graphicx,color,xcolor}
\usepackage{amsthm,amsmath,amssymb,amsfonts}
\usepackage[numbers, sort & compress]{natbib}

\def\pref#1{(\ref{#1})}

\makeatletter
\def\p@enumii{}
\makeatother

\newtheorem{thm}{Theorem}[section]
\newtheorem{lem}[thm]{Lemma}
\newtheorem{prop}[thm]{Proposition}
\newtheorem{cor}[thm]{Corollary}
\theoremstyle{definition}

\theoremstyle{remark}
\newtheorem{rem}[thm]{Remark}
\numberwithin{equation}{section}

\newcommand{\z}{\mathbb{Z}}

\newcommand{\grt}{\mathrm{girth}}
\newcommand{\V}{\mathrm{V}}
\newcommand{\m}{\mathrm{m}}
\newcommand{\im}{\mathrm{im}}
\newcommand{\E}{\mathrm{E}}
\newcommand{\N}{\mathrm{N}}
\renewcommand{\L}{\mathrm{L}}
\renewcommand{\b}[1]{\overline{#1}}
\newcommand{\rg}{\rangle}
\renewcommand{\lg}{\langle}
\newcommand{\se}{\subseteq}
\newcommand{\link}{\mathrm{link}}
\newcommand{\sm}{\setminus }

\newcommand{\give}{$\Rightarrow$}

\newcommand{\ifof}{if and only if }
\newcommand{\tohi}{\emptyset}
\newcommand{\tl}[1]{\widetilde{#1}}
\newcommand{\rnd}{\partial}
\newcommand{\blt}{\bullet}

\newcommand{\chr}{\mathrm{char}}

\title{Gorenstein and Cohen-Macaulay Matching Complexes}
\author{Ashkan Nikseresht \\
\it\small Department of Mathematics, College of Science, Shiraz University, \\
\it\small 71457-13565, Shiraz, Iran\\
\it\small E-mail: ashkan\_nikseresht@yahoo.com}
\date{}

\begin{document}
\maketitle

\begin{abstract}
Let $H$ be a simple undirected graph. The family of all matchings of $H$ forms a simplicial complex called the
matching complex of $H$. Here , we give a classification of all graphs with a Gorenstein matching complex. Also we
study when the matching complex of $H$ is Cohen-Macaulay and, in certain classes of graphs, we fully characterize
those graphs which have a Cohen-Macaulay matching complex. In particular, we characterize when the matching complex
of a graph with girth at least 5 or a complete graph is Cohen-Macaulay.
\end{abstract}


\section{Introduction}
In this paper, $K$ denotes a field and $S=K[x_1,\ldots, x_n]$. Let $G$ be a simple graph on vertex set
$\V(G)=\{v_1,\ldots, v_n\}$ and edge set $\E(G)$. Then the \emph{edge ideal} $I(G)$ of $G$ is the ideal of $S$
generated by $\{x_ix_j|v_iv_j\in \E(G)\}$. A graph $G$ is called Cohen-Macaulay (resp. Gorenstein) when $S/I(G)$ is
Cohen-Macaulay (resp. Gorenstein) for every field $K$. Many researchers have tried to combinatorially characterize
Cohen-Macaulay (CM, for short) or Gorenstein graphs in specific classes of graphs, see for example, \cite{obs to
shell,tri-free, alpha=3, planar goren, very well, hibi, CWalker}).

Note that the family of independent sets of $G$ is a simplicial complex which we denote by $\Delta(\b G)$, and is
called the \emph{independence complex} of $G$. Suppose that $I_\Delta\se S$ is the Stanley-Reisner ideal of a
simplicial complex $\Delta$. Then $\Delta$ is called CM (resp. Gorenstein) when $S/I_\Delta$ is CM (resp. Gorenstein)
for every field $K$. Noting that $I_{\Delta(\b G)}=I(G)$, we see that $G$ is CM or Gorenstein \ifof $\Delta(\b G)$ is
so.

Assume that $H$ is a simple undirected graph and $G=\L(H)$ is the \emph{line graph} of $H$, that is, edges of $H$ are
vertices of $G$ and two vertices of $G$ are adjacent if they share a common endpoint in $H$. Line graphs are well-known
in graph theory and have many applications. In particular, there are some well-known characterizations of line graphs
(see for example \cite[Section 7.1]{west}). Also a linear time algorithm is presented by Lehot in \cite{lehot} that
given a graph $G$, checks whether $G$ is a line graph and if $G$ is a line graph, returns a graph $H$ such that
$G=\L(H)$. Note that faces of $\Delta(\b G)$ are exactly the matchings of $H$, hence $\Delta(\b G)$ is known and
studied as the \emph{matching complex} of $H$ in the literature (see, for example, \cite{Match Rev, ManifoldMatching,
jonsson} and the references therein). In the survey article \cite{Match Rev}, one can find many known properties of
matching complexes of complete and complete bipartite graphs. In particular, it is known that the matching complex of a
complete bipartite graph with part sizes $m$ and $n$  with $m\leq n$ is CM \ifof $n\geq 2m-1$.

Here, first in Section 2, we characterize graphs with a Gorenstein matching complex. Then in Section 3, we focus on the
Cohen-Macaulay property and give classifications of graphs with a CM matching complex which are in one of the following
classes: Cameron-Walker graphs, graphs with girth at least 5, complete graphs. Note that, in other words, we present
characterizations of Gorenstein line graphs and CM line graphs in the aforementioned classes.

For definitions and basic properties of simplicial complexes and graphs one can see \cite{hibi} and \cite{west},
respectively. In particular, all notations used in the sequel whose definitions are not stated, are as in these
references.

\section{Gorenstein Matching Complexes}

In what follows, $H$ is a simple undirected graph with at least one edge, $G=\L(H)$ is the line graph of $H$ and
$\Delta=\Delta(\b G)$ is the matching complex of $H$. First we characterize when this complex is Gorenstein. It should
be mentioned that this characterization also follows from Theorem 4.3 of the recent paper \cite{ManifoldMatching} and
Theorem 5.1 in chapter II of \cite{stanley}, though, the proof presented here is quite different and does not need
homological concepts. Here if $F\se \V(G)=\E(H)$, then by $\N[F]$ (or $N_G[F]$) we mean $F\cup \{v\in \V(G)| uv\in
\E(G) \text{ for some } u\in F\}$, and we set $G_F=G\sm \N[F]$ and $H_F=H\sm \V(F)$, where $\V(F)$ is the set of all
vertices of $H$ incident to an edge in $F$. It is routine to check that $G_F=\L(H_F)$.
\begin{thm}\label{inde gor}
Suppose that $H$ is a graph with at least one edge, $G=\L(H)$ and $K$ be an arbitrary field. Then $G$ (or equivalently,
the matching complex of $H$) is Gorenstein over $K$ \ifof each connected component of $H$ is either a 5-cycle or a path
with length $\leq 2$ or the graph in Figure \ref{fig4}.
\end{thm}
\begin{figure}
\begin{center}
\includegraphics[scale=2.2]{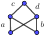}
\caption{A graph whose line graph has a Gorenstein independence complex. \label{fig4}}
\end{center}
\end{figure}
\begin{proof}
Note that $G$ is Gorenstein \ifof each connected component of $G$ is so. Since each connected component of $G$ is the
line graph of a connected component of $H$, we assume that $H$ is connected. The line graph of the stated graphs are
either $K_2$ or $K_1$ or the complement of cycles of length 5 or 6 and hence are Gorenstein by \cite[Corollary
2.4]{alpha=3}).

Conversely assume that $G$ is Gorenstein. If $\dim \Delta(\b G)=0$, then $G$ is a Gorenstein complete graph and thus
either $G$ is a vertex and hence $H$ is $K_2$ or $G$ is $K_2$ and hence $H$ is a path of length 2. Suppose $\dim
\Delta(\b G)\geq 1$. Then by \cite[Theorem 2.3]{alpha=3}, for each maximum size matching $M$ of $H$, if $F=M\sm\{a,b\}$
for some $a,b\in M$, then $G_F$ is the complement of a cycle of length at least 4. In particular, as $ab\in
\E(\b{G_F})$ there are edges $c,d\in \E(H_F)=\V(G_F)$ such that $ac, bd\in\E(G)$ but $bc,ad\notin\E(G)$. Two cases are
possible; either $c$ and $d$ have a common endpoint in $H$ or $c$ and $d$ do not have a common endpoint in $H$. In the
latter case, $c$ and $d$ are adjacent in $\b {G_F}$ and hence $\b {G_F}$ is cycle of length 4 and $H_F$ is the disjoint
union of two paths of length 2. In the former case, any other edge of $H_F$ is adjacent to both $a$ and $b$ in $G$
(since in the cycle $\b{G_F}$, degree of each vertex is 2, and $a$  and $b$ are already adjacent to two vertices among
$a,b,c,d$), therefore in $H_F$ any edge except $a,b,c,d$ must be incident to an endpoint of $a$ and also to an endpoint
of $b$. There are 4 possible such edges and hence only 16 possibilities exist for $H_F$. Because $\b{G_F}$ is also the
complement of a cycle, it follows that the only possible graphs for $H_F$ is the 5-cycle or the graph in Figure
\ref{fig4}.

Note that in both cases exactly one endpoint of each of $a$ and $b$ is adjacent to a vertex of $H$ not in $\V(M)$. It
follows that if the set of all vertices of $H$ not covered by $M$ is $\{x_1,\ldots,x_t\}$ and $M=\{u_1v_1,\ldots,
u_rv_r\}$ for $u_i,v_i\in \V(H)$, then for each $1\leq i\leq r$, exactly one of $u_i$ or $v_i$ is adjacent to exactly
one of the $x_j$'s. We can assume that $u_1,\ldots, u_{r_1}$ are the vertices adjacent to $x_1$, $u_{r_1+1},\ldots,
u_{r_1+r_2}$ are adjacent to $x_2$ and \ldots. Suppose $t>1$. Note that $H$ is connected and there is no edge between
$x_i$'s, because $M$ is a maximum matching. Thus there should be an edge with one end on some $u_iv_i$ for $i>r_1$ and
one end on $u_jv_j$ for $j\leq r_1$. But then setting $a=u_iv_i$ and $b=u_jv_j$, we see that $\b{G_F}$ is not a cycle
for $F=M\sm \{a,b\}$. Therefore, $t=1$.

Suppose that $\dim \Delta(\b G)>1$, $M$ is a maximum matching of $H$ and $a,b,c\in M$. Let $F=M\sm \{a,b,c\}$. Then for
each $l\in \{a,b,c\}$, $H_{F\cup\{l\}}$ is either a 5-cycle or the graph in Figure \ref{fig4}. It follows that, $H_F$
is one of the graphs in Figure \ref{fig5}. If $H$ is any of the first three graphs from the left in Figure \ref{fig5}
and $e$ is the edge specified in the figure, then $H_{F\cup \{e\}}$ is neither a 5 cycle nor the graph in Figure
\ref{fig4}, although, $F\cup\{e\}$ is a matching with size $\dim \Delta(\b G)-1$.

\begin{figure}
\begin{center}
\includegraphics[scale=.8]{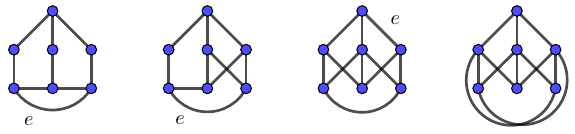}
\caption{Illustrations for the proof of Theorem \ref{inde gor}. \label{fig5}}
\end{center}
\end{figure}

Consequently, $H_F$ must be the rightmost graph depicted in Figure \ref{fig5}. Suppose that $I(G,x)$ is the
independence polynomial of $G$, that is, the coefficient of $x^n$ is the number of independent sets of $G$. Then one
can compute that $I(G_F,x)= 1+12x+36x^2+24x^3$. It follows that $I(G_F,-1)=1\neq -1=(-1)^{\dim \Delta(\b {G_F})+1}$ and
hence $G_F$ is not Gorenstein by \cite[Theorem 2.3]{alpha=3}. But according to \cite[Corollary 2.4]{Trung}, if $G$ is
Gorenstein, then $G_F$ is Gorenstein for every independent set $F$ of $G$, a contradiction. From this contradiction we
deduce that $\dim \Delta(\b G)= 1$ and as argued above, $H$ is either a 5-cycle or the graph in Figure \ref{fig4}.
\end{proof}

\section{Cohen-Macaulay Matching Complexes}

In this section we classify graphs which have a CM matching complex. We do this only when the graph is in one of the
following classes of graphs: Cameron-Walker graphs and graphs with girth at least 5 (Subsection 3.1) and complete
graphs (Subsection 3.2). Recall that here $H$ is a simple undirected graph with at least one edge, $G=\L(H)$ is the
line graph of $H$ and $\Delta=\Delta(\b G)$ is the matching complex of $H$.
\subsection{Cameron-Walker graphs and graphs with large girth}

Recall that an induced matching of $H$ is a matching $M$ of $H$ such that for every pair of edges $e\neq e'\in M$,
there is no edge $f$ of $H$ such that $f$ is incident to an endpoint of $e$ and an endpoint of $e'$. For studying
Cohen-Macaulayness of matching complexes, we first restrict ourselves to the class of graphs, such as $H$, in which the
maximum size of induced matchings (denoted by $\im(H)$) and the maximum size of matchings (denoted by $\m(H)$) are
equal. A classification of such graphs was first stated in \cite{CWalker orig}. Then in \cite{CWalker} a minor
correction to this classification was presented.

We recall the corrected classification (\cite[pp. 258, 259]{CWalker}). A \emph{star triangle} is a graph obtained by
joining several triangles at a common vertex and a \emph{pendant triangle} is a triangle with two degree 2 vertices and
a vertex with degree more than 2. Also a \emph{leaf edge} is an edge incident to a leaf. Now for a connected graph $H$,
$\im(H)=\m(H)$ \ifof either $H$ is a star or a star triangle or is obtained from a connected bipartite graph with
vertex partitions $X$ and $Y$, by adding some and at least one leaf edge to each vertex of $X$ and adding some and
maybe zero pendant triangles to each vertex  of $Y$. As in \cite{CWalker}, we call a connected graph $H$ with
$\im(H)=\m(h)$ and which is not a star or a star triangle, a \emph{Cameron-Walker graph}. By a bipartition $(X,Y)$ of a
Cameron-Walker graph, we mean a bipartition of the underlying bipartite graph where $X$ and $Y$ satisfy the
aforementioned conditions. Figure \ref{cw} shows an example of a Cameron-Walker graph.

\begin{figure}
\begin{center}
\includegraphics[scale=0.7]{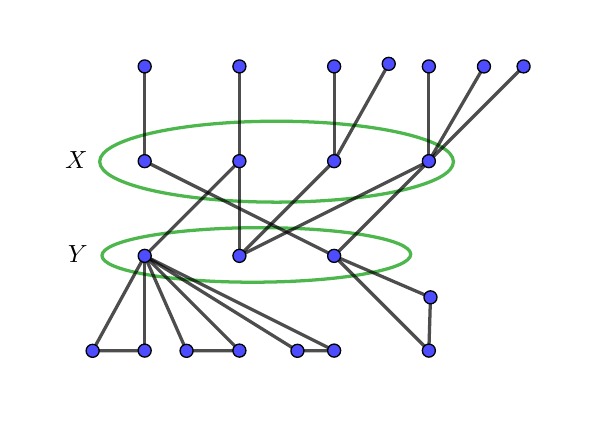}
\caption{An example of a Cameron-Walker graph \label{cw}}
\end{center}
\end{figure}

It is well-known that a CM simplicial complex is pure and the following shows that for graphs with $\im(H)=\m(h)$ the
matching complex is pure.

\begin{lem}\label{pure}
If $H$ is a star or a star triangle or a Cameron-Walker graph, then $\Delta$ is pure.
\end{lem}
\begin{proof}
It is easy to see that each maximal matching of a star and a star triangle have size 1 and $(n-1)/2$, respectively,
where $n$ denotes the number of vertices of the graph. Assume that $H$ is a Cameron-Walker graph with bipartition
$(X,Y)$. Let $M$ be a maximal matching of $H$ and $M_1$ be those edges of $M$ which are in the underlying bipartite
graph and $M_2=M\sm M_1$. Suppose that $X_1$ and $Y_1$ are vertices in $X$ and $Y$, respectively, incident to some edge
in $M_1$ and $X_2=X\sm X_1$, $Y_2=Y\sm Y_1$. Then for each vertex $x\in X_2$ there should be exactly one leaf edge in
$M_2$ incident to $x$ and for $x\in X_1$, none of the leaf edges at $x$ are in $M_2$. Also for each vertex $y\in Y$ the
pendant triangles at $y$ form a star triangle. If $y\in Y_2$, then a maximal matching of the star triangle at $y$ must
be in $M_2$ and its size as mentioned before equals the number $t_y$ of triangles in the star triangle at $y$. If $y\in
Y_1$, then as $y$ is covered by $M_1$, only those edges of the star triangle can be in $M_2$, that do not meet $y$ and
since $M$ is maximal, all of them must be in $M_2$. Therefore again the star triangle at $y$ contributes exactly $t_y$
edges to $M_2$. Consequently, $|M_2|=|X_2|+t$, where $t$ is the total number of pendant  triangles in $H$. Because
$|M_1|=|X_1|$, it follows that $|M|=|X|+t$ is independent of the choice of $M$.
\end{proof}

As a partial converse we have:
\begin{lem}\label{pure converse}
If $H$ is connected, $\Delta$ is pure and $\grt(H)\geq 5$, then either $H$ is a 5-cycle or a 7-cycle or a star or a
Cameron-Walker graph.
\end{lem}
\begin{proof}
By \cite[Theorem 1]{equimathcable}, $H$ is either a 5-cycle, a 7-cycle, an edge or a bipartite graph with bipartite
sets $V_1$ and $V_2$ such that each vertex in $V_1$ is adjacent to a leaf and each vertex in $V_2$ is not adjacent to
any leaf. Suppose that $H$ is not a star, a 5-cycle or a 7-cycle. If $L$ is the set of leaves of $H$, then by setting
$X=V_1$ and $Y=V_2\sm L$, then $H$ is a Cameron-Walker graph with bipartition $(X,Y)$. Note that the edges between $X$
and $L$ are the leaf edges at vertices of $X$.
\end{proof}

Recall that for a face $F$ of a simplicial complex $\Gamma$, we define $\link_\Gamma (F)= \{G\sm F|F\se G\in \Gamma\}$.
Also for a vertex $v$ of $\Gamma$, $\Gamma-v$ is the simplicial complex with faces $\{F\in \Gamma|v\notin F\}$. A
vertex $v$ of a nonempty simplicial complex $\Gamma$ is called a \emph{shedding vertex}, when no face of
$\link_\Gamma(v)$ is a facet of $\Gamma-v$.  A nonempty simplicial complex $\Gamma$ is called \emph{vertex
decomposable}, when either it is a simplex or there is a shedding vertex $v$ such that both $\link_\Gamma (v)$ and
$\Gamma- v$ are vertex decomposable. One should note that in some papers $v$ is called a shedding vertex when it is
shedding in the sense defined here and also $\link_\Gamma (v)$ and $\Gamma- v$ are vertex decomposable. The
$(-1)$-dimensional simplicial complex $\{\emptyset\}$ is considered a simplex and hence is vertex decomposable.

\begin{thm}\label{CM main}
Assume that $\im(H)=\m(H)$ and $G=\L(H)$. Then $\Delta(\b G)$ is vertex decomposable.
\end{thm}
\begin{proof}
By \cite{CWalker}, $H$ is either a star or a star triangle or a  Cameron-Walker graph. If $H$ is a star then $G$ is a
complete graph and by \cite[Lemma 4]{obs to shell}, $\Delta=\Delta(\b G)$ is vertex decomposable.

For the case that $H$ is a star triangle or a Cameron-Walker graph, we use induction on the number of edges of $H$.
Suppose that $H$ is a star triangle in which the common vertex of the triangles is $v$. Let $e$ be an edge incident to
$v$. Then $\link_\Delta(e)= \Delta(\b{\L(H-\N_G[e])})$ and $H-\N_G[e]$ is just a matching. Therefore $\link_\Delta(e)$
is a simplex. Also $\Delta-e=\Delta(\b{\L(H-e)})$ and $H-e$ is either a star or a Cameron-Walker graph with less edges
that $H$ and hence $\Delta-e$ is vertex decomposable by the induction hypothesis. Also if $M$ is a maximal matching of
$H$ containing $e$, then $(M\sm \{e\})\cup\{e'\}$ is a matching of $H-e$, where $e'$ is the edge in the same triangle
as $e$ but not adjacent to $v$. This shows that $e$ is a shedding vertex of $\Delta$ and thus $\Delta$ is vertex
decomposable.

Now assume that $H$ is a Cameron-Walker graph with bipartition $(X,Y)$. Let $e$ be an edge of the underlying bipartite
graph with an endpoint $x\in X$. By the definition of Cameron-Walker graphs, $x$ is incident to a leaf edge $e'$. Now
every connected component of $H-e$ and $H-N_G[e]$ is either a star or a star triangle or a Cameron-Walker graph with
less edges than $H$. Consequently, both $\Delta-e$ and $\link_\Delta(e)$ are vertex decomposable by induction
hypothesis. Also if $M$ is a maximal matching of $H$ containing $e$, then $(M\sm \{e\})\cup\{e'\}$ is a matching of
$H-e$, that is, $e$ is a shedding vertex of $\Delta$. Thus $\Delta$ is vertex decomposable.
\end{proof}

A combinatorial property related to being CM is shellability. If there is an ordering $F_1, \ldots, F_t$ of all facets
of $\Gamma$ such that for all $i$ we have $\lg F_1, \ldots, F_i \rg \cap F_{i+1}$ is a pure simplicial complex of
dimension $= \dim F_{i+1}-1$, $\Gamma$ is called \emph{shellable}. It is well-known that a vertex decomposable complex
is shellable and a pure shellable complex is CM.
\begin{cor}\label{grt>=5}
Suppose that $H$ is a connected graph with $\grt(H)\geq 5$, $G=\L(H)$ and $\Delta= \Delta(\b G)$ be the matching
complex of $H$. Then the following are equivalent.
\begin{enumerate}
\item \label{grt>=5 1} $\Delta$ is pure vertex decomposable.
\item \label{grt>=5 2} $\Delta$ is pure shellable.
\item  \label{grt>=5 3} $\Delta$ or equivalently $G$ is CM (over some field).
\item  \label{grt>=5 4} $H$ is either a 5-cycle or a star or a Cameron-Walker graph.
\end{enumerate}
\end{cor}
\begin{proof}
It is well-known that \pref{grt>=5 1} \give\ \pref{grt>=5 2} \give\ \pref{grt>=5 3}. Suppose that $\Delta$ is CM over a
field $K$. Then $\Delta$ is pure and hence according to \pref{pure converse}, $H$ is either a 5-cycle or a 7-cycle or a
star or a Cameron-Walker graph. But by \cite[Theorem 10]{obs to shell} a 7-cycle is not CM over any field, so
\pref{grt>=5 4} follows. Now assume that $H$ is a star or a Cameron-Walker graph. Then by \pref{pure} and \pref{CM
main}, $\Delta$ is pure vertex decomposable. For the 5-cycle, the result follows from \cite[Theorem 10]{obs to shell}.
\end{proof}
\subsection{Complete graphs}
At the end of this paper, we consider the matching complex of the complete graph $K_n$ on $n$ vertices. Recall that a
CM simplicial complex $\Gamma$ is \emph{strongly connected} (or \emph{connected in codimension 1}), that is, for each
two facets $A,B$ of $\Gamma$, there is a sequence of facets $A=F_1, F_2, \ldots, F_t=B$, such that $|F_i\cap F_{i+1}|=
|F_i|-1$. We call such a sequence of facets a \emph{strong path} from $A$ to $B$. First we find complete graphs with a
strongly connected matching complex.
\begin{prop}\label{Kn connected}
Suppose that $G=\L(K_n)$ and $\Delta=\Delta(\b G)$. Then $\Delta$ is strongly connected, \ifof $n$ is odd or $n=2$.
\end{prop}
\begin{proof}
Suppose that $n$ is even. Then each maximal matching of $K_n$ has size $n/2$. If $M$ and $M'$ are maximal matchings of
$K_n$ and $|M\cap M'|=|M|-1$, then $M \cap M'$ saturates  $n-2$ vertices of $K_n$. If $u$ and $v$ are the only
unsaturated vertices, then the only maximal matching containing $M\cap M'$ is $(M\cap M') \cup \{uv\}$, so $M=M'$.
Consequently, if $K_n$ has more than one maximal matchings, which holds if $n>2$, then $\Delta$ is not strongly
connected.

Now assume that $n$ is odd. We use induction to show that there is a strong path between every two maximal matchings of
$K_n$. Let $M_1$ and $M_2$ be maximal matchings of $K_n$. If $uv\in M_1 \cap M_2$ then there is a strong path between
$M_1\sm \{uv\}$ and $M_2\sm \{uv\}$ in $K_n-u-v \cong K_{n-2}$, by the induction hypothesis. Hence there is a strong
path between $M_1$ and $M_2$. Suppose that $M_1\cap M_2 =\tohi$ and $a$ is a vertex saturated by both $M_1$ and $M_2$,
say $ab\in M_1$ and $ac\in M_2$. If $c$ is not saturated by $M_1$, then by the previous part there is a strong path
between $(M_1\sm\{ab\}) \cup\{ac\}$ and $M_2$. Thus it follows that there is such a path between $M_1$ and $M_2$. If
$c$ is saturated by $M_1$, say $cd\in M_1$, then there is a strong path $P$ from $(M_1\sm\{ab,cd\}) \cup\{ac,bd\}$ to
$M_2$. Note that since $n$ is odd, there is an vertex $e$ unsaturated in $M_1$. If we attach the following strong path
at the start of $P$ we get a strong path from $M_1$ to $M_2$: $M_1$, $(M_1\sm \{ab\}) \cup\{ae\}$, $(M_1\sm \{ab,cd\})
\cup \{ae, bd\}$.
\end{proof}
Recall that a \emph{perfect matching} of $H$ is a matching that saturates all vertices of $H$.
\begin{rem}\label{per match}
The first part of proof of \pref{Kn connected}, indeed shows that if $H$ is a graph with a perfect matching, then its
matching complex cannot be strongly connected and hence is not CM.
\end{rem}

To find out when $\Delta(\b{\L(K_n)})$ is CM, we need an algebraic tool. Let $\Gamma$ be a simplicial complex and
denote by $\tl{C}_d(\Gamma)= \tl{C}_d(\Gamma; K)$ the free $K$-module whose basis is the set of all $d$-dimensional
faces of $\Gamma$. Consider the $K$-map $\rnd_d: \tl{C}_d(\Gamma) \to \tl{C}_{d-1}(\Gamma)$ defined by
 $$\rnd_d(\{v_0, \ldots, v_d\}) =\sum_{i=0}^d (-1)^i \{v_0, \ldots, v_{i-1},v_{i+1},\ldots, v_d\},$$
where $v_0<\cdots <v_d$ is a linear order. Then $(\tl{C}_\blt, \rnd_\blt)$ is a complex of free $K$-modules and
$K$-homomorphisms called the \emph{augmented oriented chain complex} of $\Gamma$ over $K$. We denote the $i$-th
homology of this complex by $\tl{H}_i(\Gamma; K)$. The Reisner theorem states that  $\Gamma$ is CM over $K$, \ifof  for
all faces $F$ of $\Gamma$ including the empty face and for all $i< \dim \link_\Gamma(F)$, one has
$\tl{H}_i(\link_\Gamma(F); K)=0$ (see \cite[Theorem 8.1.6]{hibi}).  Equivalently, this theorem states that $\Gamma$ is
CM over $K$ \ifof for all vertices $v$ of $\Gamma$, $\link_\Gamma(v)$ is CM and $\tl{H}_i(\Gamma;K)=0$ for all $i<\dim
\Gamma$.

\begin{thm}
Suppose that $G=\L(K_n)$ and $\Delta=\Delta(\b G)$ is the matching complex of $K_n$. Then $\Delta$ is CM over a field
$K$, \ifof either
\begin{enumerate}
 \item $n \leq 3$ or $n=5$;
 \item or $n=7$ and $\chr K\neq 3$.
\end{enumerate}
\end{thm}
\begin{proof}
If $n\leq 3$, then $\dim \Delta \leq 0$ and $\Delta$ is CM. If $n>2$ is even, then by \pref{Kn connected} $\Delta$ is
not strongly connected and hence not CM. If $n=5$, then $\dim \Delta=1$ and $\Delta$ is CM \ifof it is connected or
equivalently strongly connected (see \cite[Excercise 5.1.26]{CM ring}). Hence again the result follows from \pref{Kn
connected}. Suppose that $n\geq 7$ is odd. Assume that $e$ is a vertex of $\Delta$, that is, $e=uv$, for vertices $u$
and $v$ of $K_n$. Then $\link_\Delta(e)$ is the matching complex of $K_n-u-v\cong K_{n-2}$. Thus if we show that for
$n=9$, $\Delta$ is not CM over any field and for $n=7$, $\Delta$ is CM exactly when $K$ has characteristic $\neq 3$,
then the proof is concluded. But letting $n=7$, for all vertices $e$ of $\Delta$, $\link_\Delta(e)$ is CM. Hence
$\Delta$ is CM \ifof $\tl{H}_i(\Delta;K)=0$ for all $i<\dim \Delta=2$. By  Table 6.1 of \cite{Match Rev}, we have
$\tl{H}_1(\Delta;\z)=\z_3$ and $\tl{H}_0(\Delta; \z)=0$ and hence by the universal coefficient theorem,
$\tl{H}_1(\Delta;K)=\z_3\otimes K$ is zero \ifof $\chr K\neq 3$. Similar argument using Table 6.1 of \cite{Match Rev}
shows that if $n=9$, then $\tl{H}_2(\Delta, K)\neq 0$ for every field $K$ and $\Delta$ is not CM over any field.
\end{proof}

\paragraph{Acknowledgement}
The author would like to thank M. Bayer and B. Goeckner for mentioning some interesting references that the author was
not aware of and contained  answers to some questions posed in a previous version of this paper.

                            
\end{document}